\documentclass{article}     

\usepackage{amsfonts}
\usepackage{amsmath}
\usepackage{graphicx}
\usepackage{amsfonts} \usepackage{amsmath, amsthm}
\usepackage{epsfig} \usepackage{verbatim}
\usepackage{amssymb, amsthm, amsmath, bm, tikz, enumerate, color}
\usepackage{cite, authblk}

\newtheorem{theorem}{Theorem}

\newtheorem{definition}{Definition}

\newtheorem{lemma}{Lemma}

\newtheorem{proposition}{Proposition}

\newcounter{tmp}

\title
{Logarithmic scaling of planar random walk's local times}

\author[*]{P\'{e}ter N\'{a}ndori}
\author[**]{Zeyu Shen}
\affil[*]{Department of Mathematics, University of Maryland}
\affil[**]{Courant Institute, New York University}
\date{\vspace{-5ex}}

\begin{document}

\maketitle

\begin{abstract}
We prove that the local time process of a planar simple random walk, when time is scaled logarithmically, converges to a 
non-degenerate pure jump process. The convergence takes place in the Skorokhod space with
respect to the $M1$ topology and fails to hold in the $J1$ topology.
\footnote{{\it 2010 Mathematics Subject Classification}. Primary 60F17; Secondary 60J55\\
{\it Key words and phrases}. Random walk, local time, functional limit theorem}
\end{abstract}

\section{Introduction}

Consider a simple random walk's trajectory on $\mathbb Z$ up to time $n$. 
Its local time at the origin scales with $\sqrt n$; furthermore
the local time process, rescaled by $\sqrt n$, converges to the Brownian motion's local time process at zero. 
One can also explore further similarities between the two local time
processes; see for instance the "Hausdorff dimension" and "density" in \cite{BF92}.\\
The naive implementation of 
the previous idea in two dimensions is much less fruitful. The local time at the origin scales with $\log n$, 
but the local time process, rescaled by $\log n$, converges to a (random) constant function. In other
words, no interesting properties of the random walk's local time are seen in the limit. The present paper suggests
another viewpoint to this local time process, namely the logarithmic scaling of time.\\
In Section 2, we give the basic definitions. Section 3 contains the main 
theorem, its proof is in Section 4.
Section 5 demonstrates on an example that the discrete and continuous 
local time processes share some interesting features. Namely, we compute the distribution of the second
longest excursion. Finally some remarks are given in Section 6.

\section{Definitions}

We work with simple random walk in the plane, i.e. a stochastic process defined by $S_0=(0,0)$, $S_n=\sum_{i=1}^n \xi_i$, where
$\xi_i$ are independent, identically distributed with
$$ P(\xi_i=(0,1)) = P(\xi_i=(0,-1)) = P(\xi_i=(1,0)) =P(\xi_i=(-1,0)) =1/4.$$

The local time at the origin is defined by
$$ N_n = \sum_{j=1}^n 1_{\{ S_j = 0\}}.$$

We will use the following notations:
\begin{eqnarray*}
 &\Rightarrow& \text{weak convergence}\\
 &EXP& \text{an exponentially distributed random variable with mean $1$}\\
 &UNI& \text{a uniformly distributed random variable on $[0,1]$}
\end{eqnarray*}

\section{The main result}

Our main result is 

\begin{theorem}
 \label{thm1}
Let us define the stochastic process $(L_n(t))_{t \in [0,1]}$ by
$$ L_n( \log k / \log n) = \frac{N_{k}}{\log n} \quad \text{for } 1 \leq k \leq n$$
and linear interpolations. 
Then there is a non-degenerate process $\mathcal J$ such that
$$ L_n(t) \Rightarrow \mathcal J (t) $$ 
with respect to the $M1$ topology on the space $D[0,1]$
as $n \rightarrow \infty$.
\end{theorem}


First we give the definition of the process $\mathcal J$ of Theorem \ref{thm1}. 
\begin{definition}
\label{def1}
 $\mathcal J $ is a c\`adl\`ag
process on $[0,1]$ 
of independent increments with $\mathcal J (0) = 0$ and for any $0 \leq s \leq t \leq 1$,
\begin{enumerate}
 \item $ \mathcal J(t) - \mathcal J(s) = 0$
with probability $\frac{s}{t}$ and
 \item $ \mathcal J(t) - \mathcal J(s) = \xi$ with probability $(1-\frac{s}{t})$, where
$\xi$ exponentially distributed with mean $\frac{t}{\pi}$.
\end{enumerate}
\end{definition}

We omit the proof of the consistency of Definition \ref{def1} 
since it is an elementary computation.
The following are simple consequences of the definition
\begin{enumerate}
 \item[(A)] $\mathcal J$ is a non-decreasing pure jump c\`adl\`ag process with countably many jumps accumulating at zero.
 \item[(B)] For any fixed $t \in [0,1)$, the following is true. 
$\mathcal J$ is constant on the time interval $[t,1]$ with probability $t$. 
If it is not constant, then the first jump after time $t$ has conditional density
$\frac{t}{(1-t) x^2}$ for $x \in [t,1]$.
   \item[(C)] The last jump of $\mathcal J$ is uniformly distributed on $[0,1]$.
\end{enumerate}

Before proving Theorem \ref{thm1} let us recall the following theorem of Erd\H{o}s and Taylor \cite{ET60}:

\begingroup
\setcounter{tmp}{\value{theorem}}
\setcounter{theorem}{0} 
\renewcommand\thetheorem{\Alph{theorem}}
\begin{theorem}
\label{thm:2dimw}
If $S_j$ is a $2$ dimensional SSRW with $S_0=(0,0)$, then as $n \rightarrow \infty$,
$$ \frac{1}{\log n} N_n \Rightarrow  \frac{1}{\pi} EXP. $$
\end{theorem}
\endgroup


\section{Proof of Theorem \ref{thm1}}

We have to prove the convergence of finite dimensional distributions and the tightness.


In order to keep notations simple, we only show the convergence of two dimensional marginals. 
The argument works for higher dimensions exactly the same way.
Let us thus fix $0<s<t<1$. Then, by Theorem \ref{thm:2dimw}, we have both $L_n(s) \Rightarrow \mathcal J (s)$
and $L_n(t) \Rightarrow \mathcal J (t)$ as $n \rightarrow \infty$, but we need to show joint convergence.
We will use the following lemmas

\begin{lemma}
\label{jointconv}
The joint distribution of 
$$\left( \frac{N_{\lfloor n^s \rfloor }}{\log n}, \frac{ \log \| S_{\lfloor n^s \rfloor}\|}{\log n} \right)$$
converges, as $n \rightarrow \infty$ to a pair consisting of $\frac{s}{\pi} EXP$ and the constant $s/2$.
\end{lemma}

\begin{proof}
 Note that the first coordinate converges weakly by Theorem \ref{thm:2dimw} and the second coordinate converges
in probability by the central limit theorem
(in fact, it converges almost surely by the law of iterated logarithm). 
It follows that the pair also converges weakly.
\end{proof}

\begin{lemma}
\label{lemma2dim}
 Assuming that $S_0=v \in \mathbb Z^2$, let us denote the first hitting time of the origin by $\tau_v$. Then,
as $\| v\| \rightarrow \infty$,
$$ \frac{\log \tau_v}{2\log \| v\|} \Rightarrow \frac{1}{UNI}.$$
\end{lemma}
\begin{proof}
 This follows from formula (2.16) in \cite{ET60}.
\end{proof}

Now let us condition on the value of $L_n(s)$ and fix a small $\varepsilon >0$. 
By Lemma \ref{jointconv}, 
\begin{equation}
 P \left( \frac{ \log \| S_{\lfloor n^s \rfloor}\|}{\log n} \in 
\left [\frac{s}{2} - \varepsilon, \frac{s}{2} + \varepsilon \right] \right) > 1-\varepsilon
\label{proof:1}
\end{equation}
for $n$ large enough independently of the value of $L_n(s)$.
Le $\sigma$ be the first return to the origin after time $n^s$. 
Then for any $u>s$, (\ref{proof:1}) and Lemma \ref{lemma2dim} imply that
\begin{equation}
 P \left( \frac{\log \sigma}{\log n} < u \right) \rightarrow 1 - s/u,
\label{proof:2}
\end{equation}
as $n \rightarrow \infty$.
Now if $\sigma < n^{t-\delta}$ with some fixed small $\delta$, then by applying the strong Markov property of $S_n$ and 
Theorem \ref{thm:2dimw} we conclude that the 
distribution of $L_n(t) - L_n(s)$ is approximately $\frac{1}{\pi} t EXP$. 
If $\sigma > n^{t}$, then $L_n(t) = L_n(s)$. 
Thus applying (\ref{proof:2}) for $u=t-\delta$ and $u=t$ and finally letting $\delta \rightarrow 0$ proves
the convergence of two dimensional marginals.\\

Now we prove tightness by checking the conditions of Theorem 12.12.3 in \cite{W02}.
Since our processes are non-decreasing, condition (i) is equivalent to the tightness of $L_n(1)$.
By the convergence of one dimensional marginals (i.e. by Theorem \ref{thm:2dimw}) $L_n(1)$ is tight.
Checking condition (ii) is also simple for non-decreasing processes. First note that 
since $L_n$ is supported on non-decreasing functions,
the oscillations inside $(0,1)$ (denoted by $\omega_s$ in formula (4.4) of \cite{W02}) are zero. Thus 
in order to check condition (ii), we only have to control $\bar \nu (x,0,\delta)$
and $\bar \nu (x,1,\delta)$. Thus for non-decreasing processes, condition (ii) is equivalent to the following:
for every positive $\varepsilon$ and $\eta$ there exists some positive $\delta$ such that 
$P(L_n (\delta) \geq \eta) < \varepsilon$ and $P(L_n (1) - L_n (1-\delta) \geq \eta) < \varepsilon$.
Both of these follow from the convergence of two dimensional marginals. 
We have finished the proof of Theorem \ref{thm1}.


\section{The second longest excursion}

In order to illustrate that certain interesting properties of the local time process 
are transparent after going to the limit $\mathcal J$, we investigate the length of the second longest excursion.
It is well known that with probability close to one, the longest excursion is the ongoing one and this excursion
is much longer than the total length of all other excursions (this follows from the invariance principle and the
fact that the planar Brownian motion does not return to the origin). We claim that the
length of the second longest excursion 
is roughly $n^{UNI}$. This is consistent
with consequence (C), i.e. the fact that the last jump of $\mathcal J$ before time $1$ is uniformly distributed on $[0,1]$. 
In this section, we prove the above claim.\\
Let us define $\varsigma_0=0$,
$$\varsigma_l = \min \{ k> \varsigma_{l-1}: S_k =0\} $$
and $L= \max \{ l: \varsigma_l \leq n\}$. Then 
$\rho_l = \varsigma_l - \varsigma_{l-1}$ is the length of the $l$th excursion for $1 \leq l \leq L$, while $n-\varsigma_L$ is the
length of the ongoing excursion at time $n$. As it was mentioned before,
$$ \lim_{n \rightarrow \infty} P \left(\frac{\varsigma_L}{n-\varsigma_L} > \varepsilon \right) = 0 \quad 
\text{for all $\varepsilon >0$.}$$
Now we claim the following
\begin{proposition}
\label{prop1}
As $n \rightarrow \infty$,
$$ \frac{\log \max_{1 \leq l \leq L} \{ \rho_l\} }{ \log n} \Rightarrow UNI.$$
\end{proposition}

\begin{proof}
 First, we show that 
\begin{equation}
\label{eq:prop1}
 \frac{\log \varsigma_L}{ \log n} \Rightarrow UNI.
\end{equation}
In order to prove (\ref{eq:prop1}), let us fix some $s \in (0,1)$ and observe that formula (\ref{proof:2}) is applicable with $u=1$, thus
$$ \lim_{n \rightarrow \infty} P (n^s < \varsigma_L) = 1-s,$$
which is equivalent to (\ref{eq:prop1}). 
Now we have $L$ excursions whose total length is $\varsigma_L$. Recall that the length of one excursion
is heavy tailed, namely $P(\rho_1>k) \approx 1/\log k$. In such cases, the longest excursion dominates all the others, i.e.
\begin{equation}
\label{eq:prop2}
  \lim_{n \rightarrow \infty} 
P \left( \frac{ \max_{1 \leq l \leq L} \{ \rho_l\} }{ \varsigma_L} <1- \varepsilon \right) = 0 \quad 
\text{for all $\varepsilon >0$.} 
\end{equation}
The proposition follows from (\ref{eq:prop1}) and (\ref{eq:prop2}).
\end{proof}

\section{Remarks}

\begin{enumerate}
 \item 
Theorem \ref{thm1} is of course true for much wider class of planar random walks then the simple random walk defined here
(with $1/ \pi$ being replaced by some other constant).
Note that the two crucial steps in the proof are Theorem \ref{thm:2dimw} and Lemma \ref{lemma2dim}. 
Analogous statements are clearly true for any non-degenerate planar random walks of finite range and zero expectation
(for instance Theorem 2 and Theorem 4 in \cite{DSzV08} are formulated for some deterministic random walks, but the argument 
also applies to iid random walks). Thus Theorem \ref{thm1} immediately extends to such cases. However, pursuing the most 
general case is not the intent of these short notes. 
\item Theorem \ref{thm1} can be extended to the planar 
Brownian motion by taking local time in a bounded domain, e.g. the unit ball. 
Indeed, the analogue of Theorem \ref{thm:2dimw} for Brownian motion was proved in Section 2 of \cite{DK57},
while the analogue of Lemma \ref{lemma2dim} is a consequence of the well-known exit probabilities
from annuli, see e.g. Theorem 3.18 in \cite{MP10}.
 \item The convergence in Theorem \ref{thm1} can be extended to the interval $[0,T]$ with arbitrary $T>0$. Indeed, 
the definition of $L_n(t)$ and $\mathcal J$ naturally extends to the interval $t \in [0,T]$ and the same proof applies.
As a consequence, the analogue of Theorem \ref{thm1} also holds in $D[0,\infty)$ and M1 topology by general theory (cf. 
Thm 12.9.3 in \cite{W02}).
 \item Note that Theorem \ref{thm1} fails to hold in the usual J1 topology. The main difference between J1 and M1
is that in J1 big jumps have to match, while in M1 one can approximate a big jump by a collection of nearby small jumps.
Clearly, the process $L_n(t)$ has jumps of size $1/ \log n$, many of them close to each other, thus
yielding $O(1)$ jumps in the M1 limit $\mathcal J$. The classical reference on topologies of the Skorokhod
space is \cite{S56} and a recent survey is \cite{W02}. Some interesting examples of convergence in M1 
are discussed in \cite{AT92, BC07}.
 \item 
Since the celebrated paper \cite{DK57} it is well known that the
rescaled occupation time of Markov processes on infinite phase spaces
converges to Mittag-Leffler distribution under mild assumptions
(exponential distribution is Mittag-Leffler of parameter $0$). 
The functional version of this limit theorem - when time is scaled linearly - was proved
in \cite{B71}. For example, the local time process of the one dimensional random walk, 
rescaled by $\sqrt n$, converges to the Mittag-Leffler process of parameter $1/2$ 
(i.e. the local time process of one dimensional Brownian motion).
As it was already pointed out in \cite{B71},
the natural extension of Mittag-Leffler processes to parameter $0$ is degenerate
which corresponds to the degenerate limit of our $N_n$, when rescaled linearly in time.
 \item Proposition \ref{prop1} is connected to the results of \cite{CsFRS99}. 
Let $R(n)$ denote the total length of all excursions except for the
two longest. A variant of our Proposition \ref{prop1} shows that distribution of 
$\log R(n) / \log n$ is close to $UNI$; while \cite{CsFRS99} shows 
that almost surely it is in the interval 
$[1/\log^{1+\varepsilon} \log n ,1-1/\log^{1+\varepsilon} \log n]$
for $n$ large enough. 
Note that in order to prove almost sure results, one has to subtract the two longest excursions 
(and not just the longest one).
Indeed, the walk returns (albeit rarely) to the origin. Thus when a new extra-long excursion is at
birth, it is of comparable size to the former extra-long excursion for a while.

\end{enumerate}

\section*{Acknowledgement}
The main part of our research has been completed in the framework
of the Summer Undergraduate Research Program (SURE)
at CIMS, NYU in 2014 while P.N. was a Courant Instructor at CIMS, NYU. 
The funding of Z.S. from this program is gratefully
acknowledged.

\end{document}